\documentclass[reqno,twoside]{amsart}
\usepackage{eurosym}
\usepackage{amssymb}
\usepackage{amsmath}
\usepackage{amsfonts}
\usepackage{graphicx,color}

\newtheorem{theorem}{Theorem}
\theoremstyle{plain}

\newtheorem{definition}{Definition}

\newtheorem{lemma}{Lemma}

\newtheorem{proposition}{Proposition}
\newtheorem{remark}{Remark}

\numberwithin{equation}{section}

\usepackage{fancyhdr}

\fancypagestyle{mypagestyle}{%
	\fancyhf{}
	\fancyhead[OC]{ \sc G. C. G. dos Santos, G. M. Figueiredo and L. S. Tavares}
	\fancyhead[EC]{\sc A sub-supersolution method for a class...}
	\fancyhead[OR]{\thepage}
	\fancyhead[EL]{\thepage}
}

\begin{document}

\title[]{A sub-supersolution method for a class of nonlocal system involving the $p(x)-$Laplacian operator and applications}

\maketitle
\vspace{-0.5cm}
\begin{center}
\author{\sc Gelson C.G. dos Santos}
\end{center}
\vspace{-0.2cm}
\begin{center}
Universidade Federal do Par\'{a}, Faculdade de Matem\'{a}tica, CEP: 66075-110, Bel\'{e}m-PA, Brazil
\end{center}
\vspace{0.1cm}
\begin{center}
\author{\sc Giovany M. Figueiredo}
\end{center}
\vspace{-0.2cm}
\begin{center}
Universidade de Bras\'{i}lia, Departamento de Matem\'{a}tica, CEP: 70910-900, Bras\'{i}lia-DF, Brazil
\end{center}
\vspace{0.1cm}
\begin{center}
	\author{\sc Leandro S. Tavares}
\end{center}
\vspace{-0.2cm}
\begin{center}
Universidade Federal do Cariri, Centro de Ci\^{e}ncias e Tecnologia, CEP:63048-080, Juazeiro do Norte-CE, Brazil
\end{center}

\begin{abstract}
	In the present paper we study the existence of solutions for a class of nonlocal system involving the $p(x)-$Laplacian operator. The approach is based on a new sub-supersolution result.
\end{abstract}

\section{Introduction}
In this work we are interested in the nonlocal system
$$
\left\{\begin{array}{rcl}\label{problema-(P)}
-\mathcal{A}(x,|v|_{L^{r_1(x)}})\Delta_{p_1(x)} u&=&f_1(x,u,v)|v|_{L^{q_1(x)}}^{\alpha_1(x)}+g_1(x,u,v)|v|_{L^{s_1(x)}}^{\gamma_1(x)} \;\;\mbox{in} \;\;\Omega,\\
-\mathcal{A}(x,|u|_{L^{r_2(x)}})\Delta_{p_2(x)} v&=&f_2(x,u,v)|u|_{L^{q_2(x)}}^{\alpha_2(x)}+g_2(x,u,v)|u|_{L^{s_2(x)}}^{\gamma_2(x)} \;\;\mbox{in} \;\;\Omega,\\
\vspace{.2cm}
 u=v=0\;\;\mbox{on}\;\;\partial\Omega,
\end{array}
\right. \eqno{(S)}
$$
where $\Omega$ is a bounded domain in $\mathbb{R}^{N} (N > 1)$ with  $C^{2}$  boundary, $|.|_{L^m(x)}$ is the norm of the space $L^{m(x)}(\Omega),$ $-\Delta_{p(x)}u:=-div(|\nabla u|^{p(x)-2}\nabla u)$ is the  $p(x)-$Laplacian operator, $r_i,p_i,q_i, s_i,\alpha_i,\gamma_i:\Omega\rightarrow[0,\infty), i=1,2$ are measurable functions and $\mathcal{A},f_1,f_2,g_1,g_2:\overline{\Omega}\times\mathbb{R}\rightarrow\mathbb{R}$ are continuous functions satisfying certain conditions.

In the last decades several works related to the  $p/p(x)-$Laplacian operator arose, see for instance \cite{acerbi-mingione,alves-moussaoui-tavares,Baraket-Bisci,cence-rep-virk,fan-zao, fan-regular, fan-zang, fan-zhang, Fan, Liu,mrr, pucci-zhang,pucci-xiang-zhang} and the references therein. Partial differential equations involving the $p(x)-$Laplacian arise, for instance, in nonlinear elasticity, fluid mechanics, non-Newtonian fluids and image processing. See for instance \cite{acerbi-mingione,bisci-radu-serva,chen-levine,radulescu-repov,ruzicka} and the references therein for more informations.

The nonlocal term $|.|_{L^{m(x)}}$ with the condition $p(x)=r(x)\equiv2$  is considered in the well known Carrier's equation
$$\rho u_{tt}-a(x,t,|u|_{L^{2}}^{2})\Delta u=0$$
that models the vibrations of a elastic string  when the variation of the tensions are not too small. See \cite{carrier} for more details. The same nonlocal term arises also in Population Dynamics, see \cite{chipot-chang1, chipot-lovat} and its references.

In the literature there are several works related to $(S)$  with $p(x) \equiv p$ ($p$ constant), see for instance, \cite{correa-cabada, chen-gao, correa-figueiredo, correa-lopes, DLX2, gel-giovany} and the references therein. For example in \cite{correa-lopes}, Corr\^{e}a  \& Lopes studied the system

$$ 
\left\{\begin{array}{rcl}
-\Delta u^{m}=a|v|_{L^{p}}^{\alpha}  & \mbox{in}  & \Omega,\\
-\Delta v^{n}=b|u|_{L^{q}}^{\beta}  & \mbox{in}  & \Omega,\\
\vspace{.2cm}
u=v=0 & \mbox{on} & \partial\Omega,
\end{array}
\right.
$$
and in \cite{chen-gao}  the authors studied a related system by using the Galerkin method.

In \cite{correa-figueiredo} the authors used a Theorem due to  Rabinowitz ( see \cite{Rabinowitz-pto-fixo} ) to study the problem
$$ 
\left\{\begin{array}{rcl}
-\Delta_{p_{1}} u=|v|_{L^{q_{1}}}^{\alpha_{1}}  & \mbox{in}  & \Omega,\\
-\Delta_{p_{2}} v=|u|_{L^{q_{2}}}^{\alpha_{2}}  & \mbox{in}  & \Omega,\\
\vspace{.2cm}
u=v=0 & \mbox{on} & \partial\Omega.
\end{array}
\right.
$$

In \cite{gel-giovany} the authors used an abstract result involving sub and supersolutions whose proof is based on the   Schaefer's Fixed Point Theorem to study a system whose general form is
$$
\left\{\begin{array}{rcl}
-\mathcal{A}(x,|v|_{L^{r_1(x)}})\Delta u&=&f_1(x,u,v)|v|_{L^{q_1(x)}}^{\alpha_1(x)}+g_1(x,u,v)|v|_{L^{s_1(x)}}^{\gamma_1(x)} \;\;\mbox{in} \;\;\Omega,\\
-\mathcal{A}(x,|u|_{L^{r_2(x)}})\Delta u&=&f_2(x,u,v)|u|_{L^{q_2(x)}}^{\alpha_2(x)}+g_2(x,u,v)|u|_{L^{s_2(x)}}^{\gamma_2(x)} \;\;\mbox{in} \;\;\Omega,\\
\vspace{.2cm}
u=v=0\;\;\mbox{on}\;\;\partial\Omega.
\end{array}
\right. 
$$
Specifically, they considered a sublinear
system, a concave-convex problem and a system of logistic equations.

In \cite{gel-gio-le} it was studied the scalar version of  $(S),$ that is, the problem
$$
\left\{\begin{array}{rcl}\label{problema-(P)}
-\mathcal{A}(x,|u|_{L^{r(x)}})\Delta_{p(x)} u&=&f(x,u)|u|_{L^{q(x)}}^{\alpha(x)}+g(x,u)|u|_{L^{s(x)}}^{\gamma(x)} \;\;\mbox{in} \;\;\Omega,\\
\vspace{.2cm}
u&=&0\;\;\mbox{on}\;\;\partial\Omega.
\end{array}
\right. \eqno{(P)}
$$
The authors obtained  an abstract result involving sub and supersolutions for $(P)$ that generalizes Theorem 1 of \cite{gel-giovany}. As an application of such result the authors generalized for the $p(x)$-Laplacian operator the three applications of  \cite[Theorem 1]{gel-giovany}.

The goal of this work is to prove  \cite[Theorem 2]{gel-giovany} for the $p(x)-$Laplacian and the three applications contained in the mentioned paper. Thus we provide a generalization of  \cite{gel-giovany} with respect to systems with variable exponents. Below we point the main differences and difficulties of this work when compared with \cite{gel-giovany}.

\begin{itemize}
	
\item[(i)] In \cite{gel-giovany} the authors used the homogeneity of $(-\Delta,H^{1}_{0}(\Omega) )$ and the eigenfunction associated to the first eigenvalue  to construct a subolution. Differently from the  $p-$Laplacian ($p(x)\equiv p$ constant)  the $p(x)-$Lapalcian is not homogeneous. Besides that, it can happen that the  first eigenvalue and the first eigenfunction of  $(-\Delta_{p(x)}, W_{0}^{1,p(x)}(\Omega))$ do not exist. Even if the first eigenvalue and the associated eigenfunction exist the  homogeneity, in general, does not allows to  use the first eigenfunction to construct a  subsolution. In order to avoid such problems we explore some arguments of \cite{gel-gio-le}.
	
\item[(ii)] We improve some arguments of \cite{gel-giovany} and we present weaker conditions on  $r_i,q_i, s_i,\alpha_i,\gamma_i,i=1,2.$
	
\item[(iii)] We generalize \cite[Theorem 2]{gel-giovany}  and as an application it is considered  some nonlocal problems that generalizes the three systems studied in \cite{gel-giovany}.
	
\item[(iv)] As in \cite[Theorem 2]{gel-giovany} and differently from several works that consider the nonlocal term $\mathcal{A}(x,|u|_{L^{r(x)}})$ satisfying   $\mathcal{A}(x,t)\geq a_{0}>0$ (where $a_0$ is a constant), the Theorem \ref{theorem-to-(S)} permit to study $(S)$ in the mentioned case and in situations where $\mathcal{A}(x,0)=0.$
	
\item[(v)]  The abstract result involving sub and super solutions is proved by using a different argument. It is used a Theorem due to Rabinowitz that can be found in \cite{Rabinowitz-pto-fixo} and some arguments of  \cite{gel-giovany} are improved.
	
\end{itemize}

In this work we will assume that $r_{i},p_{i},q_{i},s_{i},\alpha_{i},\gamma_{i}$ satisfy
\begin{itemize}
	\item[$(H)$] $p_{i}\in C^{1}(\overline{\Omega}),r_{i},q_{i},s_{i}\in L_{+}^{\infty}(\Omega),$ where
$$L_{+}^{\infty}(\Omega)=\left\{ m\in L^{\infty}(\Omega)\;\text{with}\; ess \inf m(x)\geq1\right\}$$
and $\alpha_{i},\gamma_{i}\in L^{\infty}(\Omega)$ satisfy
	$$1<p_{i}^{-}:=\inf_{\Omega} p_{i}(x)\leq p_{i}^{+}:=\sup_{\Omega} p_{i}(x)<N\;\;\text{and}\;\;\alpha_{i}(x),\gamma_{i}(x)\geq0\;\text{a.e in}\;\Omega,$$
\end{itemize}
for $i=1,2.$

In order to present our main result, we need some definitions. We say that the pair $(u_{1},u_{2})$ is a weak solution of $(S)$, if $u_{i}\in W_{0}^{1,p_{i}(x)}(\Omega)\bigcap L^{\infty}(\Omega)$ and
$$\int_{\Omega}|\nabla u_{i}|^{p_{i}(x)-2}\nabla u_{i}\nabla\varphi=\int_{\Omega}\left(\frac{f_{i}(x,u_{1},u_{2})|u_{j}|_{L^{q_{i}(x)}}^{\alpha_{i}(x)}}{\mathcal{A}(x,|u_{j}|_{L^{r_{i}(x)}})}+\frac{g_{i}(x,u_{1},u_{2})|u_{j}|_{L^{s_{i}(x)}}^{\gamma_{i}(x)}}{\mathcal{A}(x,|u_{j}|_{L^{r_{i}(x)}})} \right)\varphi,$$
for all $\varphi\in W_{0}^{1,p_{i}(x)}(\Omega)$ with $i,j=1,2$ and $i\neq j.$

Given $u,v\in \mathcal{S}(\Omega)$ we write  $u\leq v$ if $u(x)\leq v(x)$ a.e in  $\Omega.$ If $u\leq v$ we define $[u,v]:=\bigl\{w\in \mathcal{S}(\Omega): u(x)\leq w(x)\leq v(x)\;\text{a.e in}\;\Omega\bigl\}.$

In order to simplify the next definition we will use the notation below
$$\widetilde{f}_{1}(x,t,s)=f_{1}(x,t,s),\;\widetilde{g}_{1}(x,t,s)=g_{1}(x,t,s),\;\widetilde{f}_{2}(x,t,s)=f_{2}(x,s,t)$$
and $\widetilde{g}_{2}(x,t,s)=g_{2}(x,s,t).$

We say that the pairs  $(\underline{u}_{i},\overline{u}_{i}), i=1,2$ are  a sub-supersolution for  $(S)$ if  $\underline{u}_{i}\in W_{0}^{1,p_{i}(x)}(\Omega)\cap L^{\infty}(\Omega),$ $\overline{u}_{i}\in W^{1,p_{i}(x)}(\Omega)\cap L^{\infty}(\Omega)$  with $\underline{u}_{i}\leq\overline{u}_{i},$ $\underline{u}_{i}=0\leq\overline{u}_{i}\;\text{on}\;\partial\Omega$ and for all $\varphi\in W_{0}^{1,p_{i}(x)}(\Omega)$ with $\varphi\geq0$ the following inequalities hold
\begin{equation}\label{eq2.1}
\left\{\begin{array}{rcl}
\displaystyle\int_{\Omega}|\nabla \underline{u}_{i}|^{p_{i}(x)-2}\nabla\underline{u}_{i}\nabla\varphi\leq \displaystyle\int_{\Omega}\left(\dfrac{\widetilde{f}_{i}(x,\underline{u}_{i},w)|\underline{u}_{j}|_{L^{q_{i}(x)}}^{\alpha_{i}(x)}}{\mathcal{A}(x,|w|_{L^{r_{i}(x)}})}+\dfrac{\widetilde{g}_{i}(x,\underline{u}_{i},w)|\underline{u}_{j}|_{L^{s_{i}(x)}}^{\gamma_{i}(x)}}{\mathcal{A}(x,|w|_{L^{r_{i}(x)}})}\right)\varphi,\\
\\
\displaystyle\int_{\Omega}|\nabla \overline{u}_{i}|^{p_{i}(x)-2}\nabla\overline{u}_{i}\nabla\varphi\geq \displaystyle\int_{\Omega}\left(\dfrac{\widetilde{f}_{i}(x,\overline{u}_{i},w)|\overline{u}_{j}|_{L^{q_{i}(x)}}^{\alpha_{i}(x)}}{\mathcal{A}(x,|w|_{L^{r_{i}(x)}})}+\dfrac{\widetilde{g}_{i}(x,\overline{u}_{i},w)|\overline{u}_{j}|_{L^{s_{i}(x)}}^{\gamma_{i}(x)}}{\mathcal{A}(x,|w|_{L^{r_{i}(x)}})}\right)\varphi,
\end{array}
\right.
\end{equation}
for all $w\in[\underline{u}_{j},\overline{u}_{j}]$ where  $i,j=1,2$ with $i\neq j.$

Our main result is described below.
\begin{theorem}\label{theorem-to-(S)}
	Suppose that for $i=1,2,$ $r_{i},p_{i},q_{i},s_{i},\alpha_{i}$ and $\gamma_{i}$ satisfy $(H),$ $(\underline{u}_{i},\overline{u}_{i})$ is a sub-supersolution for $(S)$ with $\underline{u}_{i}>0\;\mbox{a.e in}\;\Omega,$ $f_{i}(x,t,s), g_{i}(x,t,s)\geq0$ in $\overline{\Omega}\times[0,|\overline{u}_{1}|_{L^{\infty}}]\times[0,|\overline{u}_{2}|_{L^{\infty}}]$ and $\mathcal{A}:\overline{\Omega}\times(0,\infty)\rightarrow\mathbb{R}$ is a continuous function with $\mathcal{A}(x,t)>0\;\mbox{in}\;\overline{\Omega}\times\big[\underline{\sigma},\overline{\sigma}\big],$
	where $\underline{\sigma}:=\min\big\{|\underline{w}|_{L^{r_{i}(x)}},i=1,2\big\},$ $\overline{\sigma}:=\max\big\{|\overline{w}|_{L^{r_{i}(x)}},i=1,2\big\},$ $\underline{w}:=\min\{\underline{u}_{i},i=1,2\}$ and $\overline{w}:=\max\{\overline{u}_{i},i=1,2\}.$
	Then $(S)$ has a weak positive solution $(u_{1},u_{2})$ with $u_{i}\in[\underline{u}_{i},\overline{u}_{i}],i=1,2.$
\end{theorem}

\section{Preliminaries: The spaces $L^{p(x)}(\Omega),$ $W^{1,p(x)}(\Omega)$ and $W_{0}^{1,p(x)}(\Omega)$}

In this section, we  point some facts regarding to the spaces $L^{p(x)}(\Omega),$  $W^{1,p(x)}(\Omega)$ and $W_{0}^{1,p(x)}(\Omega)$ that will be often used in this work.  For more information see Fan-Zhang \cite{fan-zhang} and the references therein.

Let $\Omega\subset I\!\!R^{N} ( N \geq 1)$ be a bounded domain. Given $p\in L_{+}^{\infty}(\Omega),$ we define the generalized  Lebesgue space
$$
L^{p(x)}(\Omega )=\left\{u\in\mathcal{S}(\Omega): \int_{\Omega}|u(x)|^{p(x)}dx<\infty \right\},
$$
where $\mathcal{S}(\Omega):=\biggl\{u:\Omega\rightarrow\mathbb{R}: u\; \text{is measurable}\biggl\}$.\\

We define in $L^{p(x)}(\Omega)$ the norm

$$
|u|_{p(x)}:=\inf \left\{\lambda>0;\; \int_{\Omega}\left|\frac{u(x)}{\lambda}\right|^{p(x)}dx\leq 1\right\}.
$$
The space $(L^{p(x)}(\Omega), |.|_{L^{p(x)}})$ is a Banach space.

Given $m\in L^{\infty}(\Omega),$ we define
$$
m^{+}:=ess \sup_{\Omega}m(x)\;\;\text{e}\;\;m^{-}:=ess \inf_{\Omega}m(x).
$$

\begin{proposition}\label{norm-property} Define the quantity $\rho(u): =\int_{\Omega}|u|^{p(x)}dx.$ For all $u, u_{n}\in L^{p(x)}(\Omega), n \in \mathbb{N},$  the following assertions hold
	
	\noindent {\bf (i)} Let $u\neq 0$ in $L^{p(x)}(\Omega)$, then $|u|_{L^{p(x)}}=\lambda\Leftrightarrow \rho(\frac{u}{\lambda})=1.$
	
	\noindent {\bf (ii)} If $|u|_{L^{p(x)}}< 1 \;(= 1;\;> 1),$ then $\rho(u) < 1\; (= 1;\;> 1).$
	
	\noindent {\bf (iii)} If $|u|_{L^{p(x)}}> 1,$ then $|u|_{L^{p(x)}}^{p^{-}}\leq\rho(u)\leq|u|_{L^{p(x)}}^{p^{+}}.$
	
	\noindent {\bf (iv)} If $|u|_{L^{p(x)}}< 1,$ then $|u|_{L^{p(x)}}^{p^{+}}\leq\rho(u)\leq|u|_{L^{p(x)}}^{p^{-}}.$
	
	\noindent {\bf (v)} $|u_{n}|_{L^{p(x)}}\rightarrow0\Leftrightarrow\rho(u_{n})\rightarrow0,$  and $\;|u_{n}|_{L^{p(x)}}\rightarrow\infty\Leftrightarrow\rho(u_{n})\rightarrow\infty.$
\end{proposition}

\begin{theorem}
	Let $p,q\in L_{+}^{\infty}(\Omega)$. The following statements hold
	
	\noindent {\bf (i)} If $p^{-}>1$ and $\frac{1}{q(x)}+\frac{1}{p(x)}=1$ a.e in $\Omega,$ then
	$
	\left|\int_{\Omega}uvdx\right|\leq \bigl(\frac{1}{p^{-}}+\frac{1}{q^{-}}\bigl)|u|_{L^{p(x)}}|v|_{L^{q(x)}}.
	$
	
	\noindent {\bf (ii)} If $q(x)\leq p(x),\;\text{a.e in}\;\Omega$ and $|\Omega|<\infty,$ then $L^{p(x)}(\Omega)\hookrightarrow L^{q(x)}(\Omega).$
\end{theorem}

We define the generalized Sobolev space as 
$$W^{1,p(x)}(\Omega):=\left\{u\in L^{p(x)}(\Omega):\frac{\partial u}{\partial x_{j}}\in L^{p(x)}(\Omega),j=1,...,N\right\}$$
with the norm $\|u\|_{*}=|u|_{L^{p(x)}}+\sum_{j=1}^{N}\big|\frac{\partial u}{\partial x_{j}}\big|_{L^{p(x)}}, u\in W^{1,p(x)}(\Omega).$ The space $W_{0}^{1,p(x)}(\Omega)$ is defined as the closure of  $C_{0}^{\infty}(\Omega)$ with respect to the norm $\|.\|_{*}.$

\begin{theorem} If $p^{-}>1,$ then $W^{1,p(x)}(\Omega)$ is a  Banach, separable and reflexive space.
	
\end{theorem}

\begin{proposition}\label{embeddings} Let $\Omega\subset\mathbb{R}^{N}$ be a bounded domain and consider $p,q\in C(\overline{\Omega}).$ Define the function $p^{*}(x)=\frac{Np(x)}{N-p(x)}$ if $p(x)<N$ and $p^{*}(x)=\infty$ if $N\geq p(x).$ The following statements hold.
	
	\noindent {\bf (i)} (Poincar\'{e} inequality) If $p^{-}>1,$ then there is a constant $C > 0$ such that $|u|_{L^{p(x)}}\leq C|\nabla u|_{L^{p(x)}}$ for all $u\in W_{0}^{1,p(x)}(\Omega).$
\\
	
	\noindent {\bf (ii)} If $p^{-},q^{-} > 1$ and $q(x)<p^{*}(x)$ for all $x\in\overline{\Omega},$ the embedding
	$W^{1,p(x)}(\Omega) \hookrightarrow L^{q(x)}(\Omega)$ is continuous and compact.
\end{proposition}

From $(i)$ of Proposition \ref{embeddings}, we have that  $\| u\|:=|\nabla u|_{L^{p(x)}}$ defines a norm in  $W_{0}^{1,p(x)}(\Omega)$ which is equivalent to the norm $\|.\|_{*}.$

\begin{definition}Consider $u,v\in W^{1,p(x)}(\Omega)$. We say that $-\Delta_{p(x)}u\leq -\Delta_{p(x)}v,$ if
$$\int_{\Omega}|\nabla u|^{p(x)-2}\nabla u \nabla\varphi\leq\int_{\Omega}|\nabla v|^{p(x)-2}\nabla v \nabla\varphi,$$	
 for all $\varphi\in W_{0}^{1,p(x)}(\Omega)$ with $\varphi\geq0.$
\end{definition}
The following result is contained in \cite[Lemma 2.2]{fan-zang} and \cite[Proposition 2.3]{Fan}.
\begin{proposition}\label{PC} Consider $u,v\in W^{1,p(x)}(\Omega).$ If $-\Delta_{p(x)}u\leq-\Delta_{p(x)}v$ and $u\leq v$ on $\partial\Omega,$ (i.e., $(u-v)^{+}\in W_{0}^{1,p(x)}(\Omega)$) then $u\leq v$ in $\Omega.$ If $u,v\in C(\overline{\Omega})$ and $S=\big\{x\in\Omega: u(x)=v(x)\big\}$ is a compact set of $\Omega,$ then $S=\emptyset.$ 
\end{proposition}

\begin{lemma}\cite[Lemma 2.1]{Fan}\label{Fan}
	Let $\lambda>0$ be the unique solution of the problem
	\begin{equation}\label{probl-linear-lambda}
	\begin{aligned}
	\left\{\begin{array}{rcl}
	-\Delta_{p(x)}z_{\lambda} &=&\lambda\;\;\mbox{in} \;\;\Omega,\\
	\vspace{.2cm}
	u&=&0\;\;\mbox{on}\;\;\partial\Omega.
	\end{array}
	\right.
	\end{aligned}
	\end{equation}
Define $\rho_{0}=\frac{p^{-}}{2|\Omega|^{\frac{1}{N}}C_{0}}.$  If $\lambda\geq \rho_{0}$ then $|z_{\lambda}|_{L^{\infty}}\leq C^{*}\lambda^{\frac{1}{p^{-}-1}}$ and $|z_{\lambda}|_{L^{\infty}}\leq C_{*}\lambda^{\frac{1}{p^{+}-1}}$ if  $\lambda<\rho_{0}.$
Here  $C^{*}$ and $C_{*}$ are positive constants dependending only on $p^{+},p^{-},N,|\Omega|$ and $C_{0},$ where  $C_{0}$ is the best constant of the embedding  $W_{0}^{1,1}(\Omega) \hookrightarrow L^{\frac{N}{N-1}}(\Omega)$.
\end{lemma}
Regarding the function  $z_{\lambda}$ of the previous result, it follows from  \cite[Theorem 1.2]{fan-regular} and  \cite[Theorem 1]{fan-zang} that   $z_{\lambda} \in C^{1}(\overline{\Omega}) $ with $z_\lambda >0$ in $\Omega.$

The proof of Theorem \ref{theorem-to-(S)} is mainly based on the following result due to Rabinowitz:

\begin{theorem}\cite{Rabinowitz-pto-fixo}\label{Rabinowitz} Let $E$ be a Banach space and  $\Phi:\mathbb{R}^{+}\times E\rightarrow E$ a compact map such that $\Phi (0,u) = 0$ for all $u \in E.$ Then the equation
	$$u=\Phi(\lambda,u)$$
possesses an unbounded continuum  $\mathcal{C}\subset\mathbb{R}^{+}\times E$ of solutions with $(0,0)\in \mathcal{C}.$
\end{theorem}

We point out that a mapping $\Phi: E \rightarrow E$ is compact if it is continuous and for each bounded subset  $U \subset E$ one has that $\overline{\Phi(U)}$ is compact.

\section{Proof of Theorem \ref{theorem-to-(S)}} 

In this section we will prove Theorem \ref{theorem-to-(S)}. 

\begin{proof}[Proof of Theorem \ref{theorem-to-(S)}] For each $i=1,2$ consider the operators $T_{i}:L^{p_{i}(x)}(\Omega)\rightarrow L^{\infty}(\Omega)$ defined by
	$$T_{i}z(x)=\left\{\begin{array}{rcl}
	\underline{u}_{i}(x)\;\;\;\;\;\;\mbox{if}\;\;\;\;\;\;\;\;\;\;\;\;\;z(x)\leq\underline{u}_{i}(x),\\
	\vspace{.2cm}
	z(x)\;\;\;\;\;\;\mbox{if}\;\;\underline{u}_{i}(x)\leq z(x)\leq\overline{u}_{i}(x),\\
	\vspace{.2cm}
	\overline{u}_{i}(x)\;\;\;\;\;\;\mbox{if}\;\;\;\;\;\;\;\;\;\;\;\;\;z(x)\geq\overline{u}_{i}(x)
	\end{array}
	\right.$$
Since $T_{i}z\in[\underline{u}_{i},\overline{u}_{i}]$ and $\underline{u}_{i},\overline{u}_{i}\in L^{\infty}(\Omega)$ it follows that the operators $T_{i}$ are well-defined.

Let $p'_{i}(x)=\frac{p_{i}(x)}{p_{i}(x)-1}$ and consider the operators  $H_{i}:[\underline{u}_{1},\overline{u}_{1}]\times[\underline{u}_{2},\overline{u}_{2}]\rightarrow L^{p'_{i}(x)}(\Omega)$ defined by
$$H_{i}(u_{1},u_{2})(x)=\frac{f_{i}(x,u_{1}(x),u_{2}(x))|u_{j}|_{L^{q_{i}(x)}}^{\alpha_{i}(x)}}{\mathcal{A}(x,|u_{j}|_{L^{r_{i}(x)}})}+\frac{g_{i}(x,u_{1}(x),u_{2}(x))|u_{j}|_{L^{s_{i}(x)}}^{\gamma_{i}(x)}}{\mathcal{A}(x,|u_{j}|_{L^{r_{i}(x)}})}$$
where $i,j=1,2$ with $i\neq j$ and $|.|_{L^{m(x)}}$ denotes the norm of the space $L^{m(x)}(\Omega).$

Consider in the Banach space $L^{p_{1}(x)}(\Omega)\times L^{p_{2}(x)}(\Omega)$ the norm
$$|(u,v)|_{1,2}=|u|_{L^{p_{1}(x)}}+|v|_{L^{p_{2}(x)}},\;(u,v)\in L^{p_{1}(x)}(\Omega)\times L^{p_{2}(x)}(\Omega).$$

Since $f_{i},g_{i},\mathcal{A}$ are continuous functions, $\mathcal{A}(x,t)>0$ in the compact set $\overline{\Omega}\times\bigl[\underline{\sigma},\overline{\sigma}],$ $T_{i}z_{i}\in[\underline{u}_{i},\overline{u}_{i}]$ for all $ z_{i}\in L^{p_{i}(x)}(\Omega),$ $\underline{u}_{i},\overline{u}_{i}\in L^{\infty}(\Omega)$ and $|w|_{L^{m(x)}}^{\theta(x)}\leq |w|_{L^{m(x)}}^{\theta^{-}}+|w|_{L^{m(x)}}^{\theta^{+}}$ for all $w\in L^{m(x)}(\Omega)$ with $\theta\in L^{\infty}(\Omega)$ it follows that there are constants  $K_{i}>0$  such that
\begin{equation}\label{limitacao-de-Hi}
|H_{i}(T_{1}z_{1},T_{2}z_{2})|\leq K_{i}
\end{equation}
for all $(z_{1},z_{2})\in L^{p_{1}(x)}(\Omega)\times L^{p_{2}(x)}(\Omega).$

Thus by Lebesgue Dominated Convergence Theorem it follows that the mappings $(z_{1},z_{2})\mapsto H_{i}(T_{1}z_{1},T_{2}z_{2})$ are continuous from $L^{p_{1}(x)}(\Omega)\times L^{p_{2}(x)}(\Omega)$ in $L^{p'_{i}(x)}(\Omega), $ $i=1,2.$

From  \cite[Theorem 4.1]{fan-zhang} the operator
$$\Phi:\mathbb{R}^{+}\times L^{p_{1}(x)}(\Omega)\times L^{p_{2}(x)}(\Omega)\rightarrow L^{p_{1}(x)}(\Omega)\times L^{p_{2}(x)}(\Omega)$$
$$(\lambda,z_{1},z_{2})\mapsto\Phi(\lambda,z_{1},z_{2})=(u_{1},u_{2}),$$
where $(u_{1},u_{2})\in W_{0}^{1,p_{1}(x)}(\Omega)\times W_{0}^{1,p_{2}(x)}(\Omega)$ is the unique solution of
$$\left\{\begin{array}{rcl}
-\Delta_{p_{1}(x)} u_{1}=\lambda H_{1}(T_{1}z_{1},T_{2}z_{2})\;\;\mbox{in}\;\;\Omega,\\
-\Delta_{p_{2}(x)} u_{2}=\lambda H_{2}(T_{1}z_{1},T_{2}z_{2})\;\;\mbox{in}\;\;\Omega,\\
\vspace{.2cm}
u=v=0\;\;\mbox{on}\;\;\partial\Omega,
\end{array}
\right.\eqno{(S_{L})}$$
is well-defined.

\textbf{Claim 1: $\Phi$ is compact.} Let $(\lambda_{n},z^{1}_{n},z^{2}_{n})\subset \mathbb{R}^{+}\times L^{p_{1}(x)}(\Omega)\times L^{p_{2}(x)}(\Omega)$ be a bounded sequence and consider $(u^{1}_{n},u^{2}_{n})=\Phi(\lambda_{n},z^{1}_{n},z^{2}_{n}).$ The definition of $\Phi$ imply that
$$\int_{\Omega}|\nabla u^{i}_{n}|^{p_{i}(x)-2}\nabla u_{n}\nabla\varphi=\lambda_{n}\int_{\Omega}H_{i}(T_{1}z^{1}_{n},T_{2}z^{2}_{n})\varphi,\;\forall\;\varphi\in \;W_{0}^{1,p_{i}(x)}(\Omega),$$
where $i,j=1,2$ and $i\neq j.$

Considering the test function $\varphi=u_{n}^{i}$ and using the boundness of  $(\lambda_{n})$ and  the inequality \eqref{limitacao-de-Hi} we get
$$\int_{\Omega}|\nabla u^{i}_{n}|^{p_{i}(x)}\leq\overline{\lambda}K_{i}\int_{\Omega}|u_{n}^{i}|$$
for all $n \in \mathbb{N}$. Here $\overline{\lambda}$ is a constant that does not depend on $n \in \mathbb{N}.$

Since $p_{i}^{-}>1$ the embedding  $L^{p_{i}(x)}(\Omega)\hookrightarrow L^{1}(\Omega)$ holds. Combining such embedding with Poincar\'{e} inequality we get
$$\int_{\Omega}|\nabla u^{i}_{n}|^{p_{i}(x)}\leq CK_{i}\|u_{n}^{i}\|$$
for all $n\in\mathbb{N}.$
Suppose that  $|\nabla u_{n}^{i}|_{L^{p_{i}(x)}}>1.$ Thus by Proposition \ref{norm-property} we have $\|u_{n}^{i}\|^{p^{-}-1}\leq CK_{i}$ for all $n\in\mathbb{N}$ where $C$ is a constant that does not depend on $n.$ Then we conclude that  $(u^{i}_{n})$ is bounded in $W_{0}^{1,p_{i}(x)}(\Omega).$
The reflexivity of  $W_{0}^{1,p_{i}(x)}(\Omega)$ and the compact embedding $W_{0}^{1,p_{i}(x)}(\Omega) \hookrightarrow L^{p_{i}(x)}(\Omega)$ provides the result.

\textbf{Claim 2: $\Phi$ is continuous.} Consider a sequence  $(\lambda_{n},z^{1}_{n},z^{2}_{n})$ in $\mathbb{R}^{+}\times L^{p_{1}(x)}(\Omega)\times L^{p_{2}(x)}(\Omega)$ converging to $(\lambda,z^{1},z^{2})$ in $\mathbb{R}^{+}\times L^{p_{1}(x)}(\Omega)\times L^{p_{2}(x)}(\Omega)$. Define   $(u^{1}_{n},u^{2}_{n})=\Phi(\lambda_{n},z^{1}_{n},z^{2}_{n})$ and $(u^{1},u^{2})=\Phi(\lambda,z^{1},z^{2}).$

Using the definition of $\Phi$ we get
\begin{equation}\label{test1}\int_{\Omega}|\nabla u^{i}_{n}|^{p_{i}(x)-2}\nabla u_{n}^{i}\nabla\varphi=\lambda_{n}\int_{\Omega}H_{i}(T_{1}z^{1}_{n},T_{2}z^{2}_{n})\varphi
\end{equation}
and
\begin{equation}\label{test2}
\int_{\Omega}|\nabla u^{i}|^{p_{i}(x)-2}\nabla u^{i}\nabla\varphi=\lambda\int_{\Omega}H_{i}(T_{1}z^{1},T_{2}z^{2})\varphi
\end{equation}
for all $\varphi\in \;W_{0}^{1,p_{i}(x)}(\Omega)$ where $i,j=1,2$ and $i\neq j.$

Considering $\varphi=(u_{n}^{i}-u^{i})$ in \eqref{test1} and \eqref{test2} and subtracting \eqref{test2} from \eqref{test1} we get

\begin{align*}
	\int_{\Omega}\bigl<|\nabla u_{n}^{i}|^{p_{i}(x)-2}\nabla u_{n}^{i}-|\nabla u^{i}|^{p_{i}(x)-2}\nabla u^{i},\nabla (u_{n}^{i}-u^{i})\bigl>=&\int_{\Omega}\lambda_{n}H(T_{1}z^{1}_{n},T_{2}z_{n}^{2})(u_{n}^{i}-u^{i}) \\
-& \int_{\Omega}\lambda H(T_{1}z^{1},T_{2}z^{2})\bigl](u_{n}^{i}-u^{i}).
\end{align*}

Using H\"{o}lder inequality we have

\begin{align*}
\left| \int_{\Omega} \bigl< |\nabla u^{i}_{n}|^{p_i (x)-2} \nabla u^{i}_{n} - |\nabla u|^{p_i (x)-2} \nabla u^{i}, \nabla (u^{i}_{n} - u) \bigl> \right| &\leq |u^{i}_{n} - u^{i}|_{p_i (x)}  \\
 \times |\lambda_n H_{i} (T_1 z^{1}_{n}, T_2 z^{2}_{n}) - & \lambda H_{i} (T_1 z^{1}, T_2 z^{2})|_{p^{\prime}_{i}(x)}
\end{align*}

The previous arguments assure us that $(u^{i}_{n})$ is bounded in $W^{1,p_i(x)}_{0}(\Omega)$. Since $\lambda_n \rightarrow \lambda$ and $H_i (T_1 z^{1}_{n}, T_2 z^{2}_{n}) \rightarrow H_i(T_1 z^1, T_2 z^2)$
in $L^{p^{\prime}_{i}(x)}(\Omega), i=1,2$ we have

$$\left| \int_{\Omega} \bigl< |\nabla u^{i}_{n}|^{p_i (x)-2} \nabla u^{i}_{n} - |\nabla u|^{p_i (x)-2} \nabla u^{i}, \nabla (u^{i}_{n} - u) \bigl>\right| \rightarrow 0. $$
Therefore $u^{i}_{n} \rightarrow u^{i}$ in $L^{p_{i}(x)}(\Omega)$ for $i=1,2$ which proves the continuity of $\Phi.$

Combining the fact  that $\Phi(0,z_{1},z_{2})=(0,0,0)$ for all $(z_1,z_2) \in L^{p_{1}(x)}(\Omega) \times L^{p_{2}(x)}(\Omega) $ with the previous claims we have by Theorem \ref{Rabinowitz} that the equation $\Phi(\lambda,u,v) = (u,v) $ possesses an unbounded continuum  $\mathcal{C}\subset\mathbb{R}^{+}\times L^{p_{1}(x)} (\Omega) \times L^{p_{2}(x)} (\Omega) $ of solutions with $(0,0,0)\in \mathcal{C}.$

\textbf{Claim 3: $\mathcal{C}$ is bounded with respect to the parameter $\lambda$.} Suppose that there exists $\lambda^* >0$ such that $\lambda \leq \lambda^*$ for all $(\lambda,u^1,u^2) \in \mathcal{C}.$ For $(\lambda,u^1,u^2) \in \mathcal{C}$ the definition of $\Phi$ imply that
$$\left\{\begin{array}{rcl}
-\Delta_{p_{1}(x)} u_{1}=\lambda H_{1}(T_{1}u_{1},T_{2}u_{2})\;\;\mbox{in}\;\;\Omega,\\
-\Delta_{p_{2}(x)} u_{2}=\lambda H_{2}(T_{1}u_{1},T_{2}u_{2})\;\;\mbox{in}\;\;\Omega,\\
\vspace{.2cm}
u_{1}=u_{2}=0\;\;\mbox{on}\;\;\partial\Omega.
\end{array}
\right.\leqno{(P_{\lambda})}$$
Using the test function $u_{i}$  in $(P_{\lambda})$ and considering \eqref{limitacao-de-Hi} we get
$$\int_{\Omega}|\nabla u_{i}|^{p_{i}(x)}\leq\lambda^{*}C|u_{i}|_{L^{p(x)}}.$$

Suppose that $|\nabla u_{i}|_{L^{p(x)}}>1.$ Then using Proposition  \ref{norm-property} and the Poincar\'{e} inequality we obtain that
$$|u_{i}|_{L^{p_{i}(x)}}^{p_{i}-1}\leq\lambda^{*}C.$$
Thus  $\mathcal{C}$ is bounded in $\mathbb{R}^{+}\times L^{p_{1}(x)}(\Omega)\times L^{p_{2}(x)}(\Omega),$ which is a contradiction.  

Considering $\lambda=1,$ by $(P_{\lambda})$ we have

\begin{equation} \label{solution-trun}
\begin{aligned}
\int_{\Omega}|\nabla u_{i}|^{p_{i}(x)-2}\nabla u_{i}\nabla\varphi&= \int_{\Omega}\left(\frac{f_{i}(x,T_{1}u_{1},T_{2}u_{2})|T_{j}u_{j}|_{L^{q_{i}(x)}}^{\alpha_{i}(x)}}{\mathcal{A}(x,|T_{j}u_{j}|_{L^{r_{i}(x)}})}\right)\varphi  \\
&+ \int_{\Omega}\left(\frac{g_{i}(x,T_{1}u_{1},T_{2}u_{2})|T_{j}u_{j}|_{L^{s_{i}(x)}}^{\gamma_{i}(x)}}{\mathcal{A}(x,|T_{j}u_{j}|_{L^{r_{i}(x)}})}\right)\varphi,
\end{aligned}
\end{equation}
for all $\varphi\in W_{0}^{1,p_{i}(x)}(\Omega)$ where $i,j=1,2$ with $i\neq j.$

We claim that
$$u_{i}\in[\underline{u}_{i},\overline{u}_{i}],\;i=1,2.$$
In order to verify the claim define
$$L_{1}(\underline{u}_{1}-u_{1})_{+}:=\int_{\{\underline{u}_{1}\geq u_{1}\}}\bigl<|\nabla \underline{u}_{1}|^{p_{1}(x)-2}\nabla\underline{u}_{1}-|\nabla u_{1}|^{p_{1}(x)-2}\nabla u_{1},\nabla(\underline{u}_{1}-u_{1})\bigl>.$$
Since $T_{2}u_{2}\in[\underline{u}_{2},\overline{u}_{2}],$ $\underline{u}_{i}(x)>0\;\mbox{a.e in}\;\Omega,$ considering $i=1, j=2,$ $w=T_{2}u_{2}$ and $\varphi=(\underline{u}_{1}-u_{1})_{+}$ in the first inequality of \eqref{eq2.1} and combining with \eqref{solution-trun} we get

\begin{eqnarray*}
	L_{1}(\underline{u}_{1}-u_{1})_{+}\displaystyle&\leq&\int_{\{\underline{u}_{1}\geq u_{1}\}}\frac{f_{1}(x,\underline{u}_{1},T_{2}u_{2})
		(|\underline{u}_{2}|_{L^{q_{1}(x)}}^{\alpha_{1}(x)}-|T_{2}u_{2}|_{L^{q_{1}(x)}}^{\alpha_{1}(x)})}{\mathcal{A}(x,|T_{2}u_{2}|_{L^{r_{1}(x)}})}(\underline{u}_{1}-u_{1})\\
	&+&\int_{\{\underline{u}_{1}\geq u_{1}\}}\frac{g_{1}(x,\underline{u}_{1},T_{2}u_{2})
		(|\underline{u}_{2}|_{L^{s_{1}(x)}}^{\gamma_{1}(x)}-|T_{2}u_{2}|_{L^{s_{1}(x)}}^{\gamma_{1}(x)})}{\mathcal{A}(x,|T_{2}u_{2}|_{L^{r_{1}(x)}})}(\underline{u}_{1}-u_{1}),
\end{eqnarray*}
that is,
$$\int_{\{\underline{u}_{1}\geq u_{1}\}}\bigl<|\nabla \underline{u}_{1}|^{p_{1}(x)-2}\nabla\underline{u}_{1}-|\nabla u_{1}|^{p_{1}(x)-2}\nabla u_{1},\nabla(\underline{u}_{1}-u_{1})\bigl>\leq0.$$
Therefore $\underline{u}_{1}\leq u_{1}.$ The same reasoning imply the other inequalities. Since $u_{i}\in[\underline{u}_{i},\overline{u}_{i}],$ we have $T_{i}u_{i}=u_{i}.$ Therefore the pair $(u_{1},u_{2})$ is a weak positive solution of $(S).$


\end{proof}

\section{Applications}

The main goal of this section is to apply Theorem \ref{theorem-to-(S)} to some classes of nonlocal problems.

\subsection{A sublinear problem:}In this section, we use Theorem \ref{theorem-to-(S)} to study the nonlocal problem
$$
\left \{
\begin{array}{rclcl}
-\mathcal{A}(x,|v|_{L^{r_1(x)}})\Delta_{p_1(x)}u &=&(u^{\beta_1(x)} + v^{\gamma_1(x)})|v|_{L^{q_1(x)}}^{\alpha_1(x)} \ \mbox{in} \   \Omega, \\
-\mathcal{A}(x,|u|_{L^{r_2(x)}})\Delta_{p_2(x)}v &=&(u^{\beta_2(x)} + v^{\gamma_2(x)})|u|_{L^{q_2(x)}}^{\alpha_2(x)} \ \mbox{in} \   \Omega, \\
u=v&=& 0\ \mbox{on}  \ \partial \Omega.
\end{array}
\right.\eqno{(S_s)}
$$

The above problem in the case $p_1(x)\equiv p_1(x) \equiv 2,$ was considered  recently in  \cite{gel-giovany}. The  result of this section generalizes \cite[Theorem 6]{gel-giovany}.

\begin{theorem}\label{teo-sublinear} Suppose that $p_i,q_i,r_i,s_i,i=1,2$ satisfy $(H)$ and let  $\alpha_i,\beta_i \in L^{\infty}(\Omega), i=1,2.$ Consider also that 
	$$0< \alpha_1^{+}+\gamma_1^{+}<p_i^{-}-1 , \ \ \ 0<\frac{\alpha^{+}_1}{p^{-}_{2}-1} + \frac{\beta^+_1}{p^{-}_{1}-1}  <1$$
and
$$ 0<   \alpha^{+}_{2} + \gamma^{+}_{2}  < p^{-}_{i}-1, \ \ \  0<\frac{\alpha^+_2}{p^{-}_{1} -1} + \frac{\beta^{+}_{2}}{p^{-}_{2}-1} < 1$$
for $i=1,2.$
Let  $a_{0}>0$ be a positive constant. Suppose that one of the conditions holds.
\vspace{0.2cm}		
	
		\noindent{\bf $(A_{1})$} $\mathcal{A}(x,t)\geq a_{0}\;\text{in}\;\overline{\Omega}\times[0,\infty), $

\vspace{0.2cm}

\noindent{\bf $(A_{2})$} $0<\mathcal{A}(x,t)\leq a_{0}\;\text{in}\;\overline{\Omega}\times(0,\infty)$ and $\lim_{t \rightarrow + \infty}\mathcal{A}(x,t)=a_{\infty}>0$ uniformly in $\Omega.$ 

\vspace{0.2cm}	

Then $(S_s)$ has a positive solution.
\end{theorem}

\begin{proof} Suppose that $(A_1)$  holds, that is, $\mathcal{A}(x,t)\geq a_{0} $ in $\overline{\Omega} \times [0,+\infty)$. We will start by constructing  $(\overline{u}, \overline{v})$. Let $\lambda>0$ and consider  $z_{\lambda}\in W_{0}^{1,p_1(x)}(\Omega)\cap L^{\infty}(\Omega)$ and $y_{\lambda}\in W_{0}^{1,p_2(x)}(\Omega)\cap L^{\infty}(\Omega)$  the unique  solutions of \eqref{probl-linear-lambda} where $\lambda$ will be chosen later. 
	
For $\lambda >0$ sufficiently large, by Lemma \ref{Fan} there is a constant $K>1$ that does not depend on $\lambda$  such that
\begin{equation}\label{desig1-p-supsol}
0<z_{\lambda}(x)\leq K\lambda^{\frac{1}{p_{1}^{-}-1}}\;\text{in}\;\Omega,
\end{equation}	
and
\begin{equation}\label{ddesig1-p-supsol}
0<y_{\lambda}(x)\leq K\lambda^{\frac{1}{p_{2}^{-}-1}}\;\text{in}\;\Omega.
\end{equation}

Since $\alpha_1^{+}+\gamma_1^{+}<p_2^{-}-1$ and $ \frac{\alpha^{+}_{1}}{p^{-}_{2}-1} + \frac{\beta^+_1}{p^{-}_{1}-1}  <1$ we can choose $\lambda>1$ such that \eqref{desig1-p-supsol} and \eqref{ddesig1-p-supsol} occur and
\begin{equation}\label{desig2-p-supsol}
\frac{1}{a_{0}}(K^{\beta_{1}^{+}}\lambda^{\frac{\beta_{1}^{+}}{p^{-}_{1}-1}  + \frac{\alpha_{1}^{+}}{p^{-}_{2}-1} } + K^{\gamma_{1}^{+}}\lambda^{\frac{\alpha_{1}^{+} + \gamma_{1}^{+}}{p^{-}_{2}-1}  } )\max\{|K|_{L^{q_1(x)}}^{\alpha^{-}},|K|_{L^{q_1(x)}}^{\alpha^{+}}\}\leq\lambda.
\end{equation}	

By \eqref{desig1-p-supsol}, \eqref{ddesig1-p-supsol} and \eqref{desig2-p-supsol}, we get $$\frac{1}{a_{0}}(z_{\lambda}^{\beta_1(x)}+ w^{\gamma_1(x)})|y_{\lambda}|_{L^{q_1(x)}}^{\alpha_1(x)}\leq   \lambda, w \in [0, y_{\lambda}] $$
Thus for $w \in [0, y_{\lambda}]$
\begin{equation*}
\begin{aligned}
\left\{\begin{array}{rcl}
-\Delta_{p_1(x)}z_{\lambda}&\geq&\dfrac{1}{\mathcal{A}(x,|w|_{L^{r_1(x)}})}(z_{\lambda}^{\beta_1(x)} + w^{\gamma_1(x)})|y_{\lambda}|_{L^{q_1(x)}}^{\alpha_1(x)}\;\;\mbox{in}\;\;\Omega,\\
\vspace{.2cm}
z_{\lambda}&=&0\;\;\mbox{on}\;\;\partial\Omega.
\end{array}
\right.
\end{aligned}
\end{equation*}

Considering, if necessary,  a larger $\lambda >0$  the previous reasoning imply 
\begin{equation*}
\begin{aligned}
\left\{\begin{array}{rcl}
-\Delta_{p_2(x)}y_{\lambda}&\geq&\dfrac{1}{\mathcal{A}(x,|w|_{L^{r_2(x)}})}(w^{\beta_2(x)} + {y_{\lambda}}^{\gamma_2(x)})|z_{\lambda}|_{L^{q_2(x)}}^{\alpha_2(x)}\;\;\mbox{in}\;\;\Omega,\\
\vspace{.2cm}
y_{\lambda}&=&0\;\;\mbox{on}\;\;\partial\Omega,
\end{array}
\right.
\end{aligned}
\end{equation*}
for all $w \in [0,z_{\lambda}].$

Now we will construct $(\underline{u},\underline{v}),i=1,2.$ Since $\partial \Omega$ is $C^2 ,$ there is a constant $\delta >0$ such that $d \in C^{2}(\overline{\Omega_{3 \delta}})$ and $|\nabla d(x)| \equiv 1,$ where $d(x):= dist(x,\partial \Omega)$ and $\overline{ \Omega_{3 \delta}}:=\{x \in  \overline{\Omega}; d(x) \leq 3 \delta\}$.
From \cite[Page 12]{Liu}, we have that,  for  $\sigma \in (0, \delta)$ suffciently small,  the function $\phi_i=\phi_i(k,\sigma),i=1,2$ defined by


\begin{equation*}
\phi_i(x)=\left\{\begin{array}{lcl}
e^{kd(x)}-1 & \text{ if } & d(x)<\sigma,\\ 
e^{k\sigma}-1+\int_{\sigma}^{d(x)}ke^{k\sigma}\Big(\frac{2\delta-t}{2\delta-\sigma}\Big)^{\frac{2}{p^{-}_{i}-1}}dt & \text{ if } &\sigma\leq d(x)<2\delta,\\
e^{k\sigma}-1+\int_{\sigma}^{2\delta}ke^{k\sigma}\Big(\frac{2\delta-t}{2\delta-\sigma}\Big)^{\frac{2}{p^{-}_{i}-1}}dt & \text{ if } & 2\delta \leq d(x),
\end{array}
\right.
\end{equation*}
belongs to  $ C^{1}_{0}(\overline{\Omega})$, where $k>0$ is an arbitrary number. They also proved that
$$-\Delta_{p_i(x)}(\mu\phi_i)=\begin{cases}
-k(k\mu e^{kd(x)})^{p_i(x)\!-1}\Big[(p_i(x)\!\!-1)+(d(x)\!\!+\frac{\ln k\mu}{k})\nabla p_i(x)\nabla d(x)\\
+\frac{\Delta d(x)}{k}\Big] \;\; \mbox{ if}\quad d(x)<\sigma,\\
\Big\{\frac{1}{2\delta-\sigma}\frac{2(p_i(x)-1)}{p^{-}_{i}-1}\!-\!\Big(\frac{2\delta-d(x)}{2\delta-\sigma}\Big)\Big[\ln k\mu e^{k\sigma}\Big(\frac{2\delta-d(x)}{2\delta-\sigma}\Big)^{\frac{2}{p^{-}_{i}-1}}\nabla p_i(x)\nabla d(x)\\
+\Delta d(x)\Big]\Big\}
(k\mu e^{k\sigma})^{p_i(x)-1}\Big(\frac{2\delta-d(x)}{2\delta-\sigma}\Big)^{\frac{2(p_i(x)-1)}{p^{-}_{i}-1}-1}\;\; \mbox{ if}\quad \sigma < d(x)<2\delta,\\
0\;\; \mbox{ if}\quad 2\delta<d(x)
\end{cases}
$$
for all $\mu >0$ and $i=1,2.$

Define $\mathcal{A}_{\lambda}:=\max\bigl\{\mathcal{A}(x,t):(x,t)\in\overline{\Omega}\times\bigl[0,\max\{|y_{\lambda}|_{L^{r_1(x)}}|z_{\lambda}|_{L^{r_2(x)}}\}\bigl]\bigl\}.$ We have $$a_{0}\leq\mathcal{A}(x,|w|_{L^{r_1(x)}})\leq \mathcal{A}_{\lambda}\;\;\mbox{in}\;\Omega $$
for all $w \in [0,y_{\lambda}].$

Let $\sigma=\frac{1}{k}\ln 2$ and $\mu=e^{-ak}$ where $$a=\frac{\min\{p^{-}_{1}-1, p^{-}_{2}-1\}}{\max\{\max_{\overline{\Omega}}|\nabla p_1|+1, \max_{\overline{\Omega}}|\nabla p_2|+1 \}}.$$
 Then $e^{k\sigma}=2$ and $k\mu \leq 1 $ if $k>0$ is sufficienltly large. 
 
Let $x \in \Omega$ with $d(x) < \sigma $. If $k>0$ is large enough we have $|\nabla d (x)| = 1$ and then we have
\begin{equation}\label{negative}
\begin{aligned}
\left|d(x) + \frac{\ln(k\mu)}{k}\right| |\nabla p_1(x)||\nabla d(x)| \leq& \left(|d(x)| + \frac{|\ln(k\mu)|}{k}\right)|\nabla p_1(x)|\\
\leq & \left(\sigma - \frac{\ln(k \mu)}{k}\right)|\nabla p_1(x)| \\
=&\left( \frac{\ln2}{k} - \frac{\ln k}{k}\right)|\nabla p_1(x)| +a |\nabla p_1(x)|\\
<& p^{-}_{1}-1.
\end{aligned}
\end{equation}
Note also that there exists a constant $A>0$ that does not depend on $k$ such that $|\Delta d (x)|  < A$ for all $x \in \overline{\partial \Omega_{3\delta}}$. Using the last inequality and the expression of $-\Delta_{p_1(x)}(\mu \phi),$ we get $-\Delta_{p_1(x)} (\mu \phi_1) \leq 0$ for $x \in \Omega$ with $ d(x) < \sigma$ or $d(x) > 2\delta$ for $k>0$ large enough. Therefore
\begin{align*}
-\Delta_{p_1(x)}(\mu\phi_1)\leq&0\leq\frac{1}{\mathcal{A}_{\lambda}}(\mu\phi_1)^{\beta_1(x)}|\mu\phi_2|_{L^{q_1(x)}}^{\alpha_1(x)}\\
\leq & \frac{1}{\mathcal{A}_{\lambda}}((\mu \phi_1)^{\beta_1(x)} + w^{\gamma_1(x)})|\mu \phi_2|_{L^{q_1(x)}}^{\alpha_1(x)}
\end{align*}
for all $w \in L^{\infty}(\Omega)$ with $w \geq \mu\phi_2$ and $d(x) < \sigma$ or $ 2\delta < d(x)$ and using the idea of the proof of estimate (3.10)  from \cite{Liu} we get
\begin{align}\label{desi1-p-subsol}
-\Delta_{p_1(x)}(\mu\phi_1)\leq&  \tilde{C}(k\mu)^{p^{-}_{1}-1}|\ln k\mu| \nonumber\\
=&  \tilde{C}(k\mu)^{p^{-}_{1}-1}\left|\ln \frac{k}{e^{ak}}\right| \;\text{ if }\;\sigma < d(x)<2\delta.
\end{align}	

From the proof of \cite[Theorem 2]{gel-gio-le} and the fact that $\alpha^{+}_{1} + \gamma^{+}_{1} < p^{-}_{1} -1$ we get
\begin{equation}\label{lhopital}
\displaystyle\lim_{k \rightarrow + \infty} \displaystyle\frac{\tilde{C} k^{p^{-}_{1}-1}}{e^{ak(p^{-}_{1}-1 - (\alpha^{+}_{1}  + \gamma^{+}_{1}))}} \left| \ln \displaystyle\frac{k}{e^{a k}}\right| = 0.
\end{equation}

 Note that $\phi_1 (x) \geq 1$ if $\sigma\leq d(x)<2\delta$ because $\phi_1(x) \geq e^{k\sigma} -1 $ and $e^{k\sigma} =2$ for all $k>0$. Thus, there  is a constant $C_0 > 0$ that does not depend on $k$ such that $| \phi_2|^{\alpha_1(x)}_{L^{q_1(x)}(\Omega)}\geq C_0$ if $\sigma <d(x)< 2\delta$. By \eqref{lhopital}, we can choose  $k>0$ large enough such that
\begin{equation}\label{desig2-p-subsol}
\frac{\tilde{C}k^{p^{-}_{1}-1}}{e^{ak[(p^{-}_{1}-1)-(\alpha^{+}_{1}+\beta^{+}_{1})]}}\Big|\ln\frac{k}{e^{a k}}\Big|\leq\frac{C_{0}}{\mathcal{A}_{\lambda}}.
\end{equation}

Therefore from \eqref{desi1-p-subsol} and \eqref{desig2-p-subsol}  we have
$$-\Delta_{p_1(x)}(\mu\phi_1)\leq \frac{1}{\mathcal{A}_{\lambda}}((\mu\phi_1)^{\beta_1(x)}+w^{\gamma_1(x)})|\mu\phi_2|_{L^{q_1(x)}}^{\alpha_1(x)},$$
for all $w\in L^{\infty}(\Omega)$ with $w \geq \mu \phi_2$ and $\sigma <d(x)<2\delta$ for $k>0 $ large enough. Therefore
$$ -\Delta_{p_1(x)}(\mu\phi_1)\leq \frac{1}{\mathcal{A}_{\lambda}}((\mu\phi_1)^{\beta_1(x)} + w^{\gamma_1(x)})|\mu \phi_2|_{L^{q_1(x)}}^{\alpha_1(x)} \ \textrm{in} \ \Omega.$$

Fix $k>0$ satisfying the above property and the inequality $- \Delta_{p_1(x)} (\mu \phi_1) \leq 1.$ For $\lambda >1 $ we have $- \Delta_{p_1(x)} (\mu \phi_1) \leq  - \Delta_{p_1(x)} z_{\lambda}$. Therefore $\mu \phi \leq z_{\lambda}.$ 

Since $ \alpha^{+}_{2} + \gamma^{+}_{2} < p^{-}_{2}-1 $, a similar reasoning imply that there is $\mu>0$ small such that
$$- \Delta_{p_{2}(x)}(\mu \phi_2)  \leq  \frac{1}{\mathcal{A}(x, |w|_{L^{r_2}(x)})} (w^{\beta_2} + (\mu\phi_2)^{\gamma_2})|\mu \phi_1|^{\alpha_2 (x)}_{L^{q_2(x)}(\Omega)} \ \textrm{in} \ \Omega$$
for all $w \in L^{\infty}(\Omega)$ with $w \geq \mu \phi_1$ and that $\mu_2 \phi \leq y_{\lambda}.$ The first part of the result is proved.

Now suppose that $0<\mathcal{A}(x,t)\leq a_{0}$ in $\overline{\Omega}\times(0,\infty).$ Let $\delta, \sigma, \mu, a, \lambda, z_{\lambda}, y_{\lambda}$ and $\phi_{i},i=1,2$ as before. From the previous arguments there exist $k>0$ large enough and  $\mu>0$ small such that
\begin{equation}\label{a_01}
-\Delta_{p_1(x)}(\mu\phi_1)\leq1 \text{,}\;\;-\Delta_{p_1(x)}(\mu\phi)\leq \frac{1}{a_{0}}((\mu\phi_1)^{\beta_1(x)}+w^{\gamma_1(x)})|\mu \phi_2|_{L^{q_1(x)}}^{\alpha_1(x)}\;\;\text{in}\;\Omega
\end{equation}
for all $w \in [\mu \phi_2,y_{\lambda}]$
and also that
\begin{equation}\label{a_02}
-\Delta_{p_2(x)}(\mu\phi_2)\leq1 \text{,}\;\;-\Delta_{p_2(x)}(\mu\phi_2)\leq \frac{1}{a_{0}}(w^{\beta_2(x)}+(\mu\phi_2)^{\gamma_2(x)})|\mu\phi_1|_{L^{q_2(x)}}^{\alpha_2(x)}\;\;\text{in}\;\Omega
\end{equation}
for all $w \in [\mu \phi_1,z_{\lambda}]$.

Since $\lim_{t\rightarrow\infty}\mathcal{A}(x,t)=a_{\infty}>0$ uniformly in $\Omega$ there is a large constant  $a_{1}>0$ such that $\mathcal{A}(x,t)\geq\frac{a_{\infty}}{2}$ em $\overline{\Omega}\times(a_{1},\infty).$ Let $$m_{k}:=\min\big\{\mathcal{A}(x,t): (x,t) \in \overline{\Omega}\times[\min\{|\mu\phi_1|_{L^{r_1(x)}},|\mu\phi_2|_{L^{r_2(x)}}\}, a_{1}] \big\}>0$$
and $\mathcal{A}_{k}:=\min\big\{m_k,\frac{a_{\infty}}{2}\big\},$ we have 
$\mathcal{A}(x,t)\geq\mathcal{A}_{k}\; \text{in}\; \overline{\Omega}\times[\min\{|\mu\phi_1|_{L^{r_1(x)}},|\mu\phi_2|_{L^{r_2(x)}}\},\infty).$

Fix $k>0$ satisfying \eqref{a_01} and \eqref{a_02}. Let $\lambda>1$ such that  \eqref{desig1-p-supsol}  and \eqref{ddesig1-p-supsol} ocurrs and 
$$\frac{1}{\mathcal{A}_{k}}\left(K^{\beta^{+}_{1}}\lambda^{\frac{\beta^{+}_{1}}{p^{-}_{1}-1} + \frac{\alpha_{1}^{+}}{p^{-}_{2}-1}} + K^{\gamma^{+}_{1}} \lambda^{\frac{\alpha^{+}_{1} + \gamma^{+}_{1}}{p^{-}_{2}-1}}\right)\max\{|K|_{L^{q_1(x)}}^{\alpha^{-}_{1}},|K|_{L^{q_1(x)}}^{\alpha^{+}_{1}}\}\leq\lambda$$
and
$$ \frac{1}{\mathcal{A}_k}\left(K^{\beta^{+}_{2}} \lambda^{\frac{\beta^{+}_{2} + \alpha^{+}_{2}}{p^{-}_{1} -1}} + K^{\gamma^{+}_{2}} \lambda^{\frac{\gamma^{+}_{2}}{p^{-}_{2} - 1} + \frac{\alpha^{+}_{2}}{p^{-}_{1}-1}}\right) \max\{|K|^{\alpha^{+}_{2}}_{L^{q_2(x)}}, |K|^{\alpha_{2}^{-}}_{L^{q_2(x)}}\} \leq \lambda$$
where $K>1$ is a constant that does not depend on $k$ and $\lambda$ (see Lemma \ref{Fan}).  Therefore we have
$$-\Delta_{p_1(x)}z_{\lambda}\leq \frac{1}{\mathcal{A}(x,|w|_{L^{r_1(x)}})}(z_{\lambda}^{\beta_1(x)} + w^{\gamma_1(x)})|y_{\lambda}|_{L^{q_1(x)}}^{\alpha_1(x)}\; \text{in}\; \Omega, w \in [\mu \phi_2, y_{\lambda}].$$
Arguing as before and considering a suitable choice for $\lambda$ and $k$ we get
$$-\Delta_{p_2(x)}y_{\lambda}\leq \frac{1}{\mathcal{A}(x,|w|_{L^{r_2(x)}})}(w^{\beta_2(x)} + y_{\lambda}^{\beta_2(x)}  )|z_{\lambda}|_{L^{q_2(x)}}^{\alpha_2(x)}\; \text{in}\; \Omega, w \in [\mu \phi_1, z_{\lambda}]$$
The comparison principle imply $\mu \phi_1 \leq z_{\lambda}$ and $\mu \phi_2 \leq y_{\lambda}$ if $\mu$ is small. The result is proved.
\end{proof}

\subsection{A concave-convex problem:}

In this section we consider the following nonlocal problem with concave-convex nonlinearities

$$
\left \{
\begin{array}{rclcl}
-\mathcal{A}(x,|v|_{L^{r_1(x)}})\Delta_{p_1(x)}u &=&\lambda |u|^{\beta_1 (x)-1}u|v|^{\alpha_1(x)}_{L^{q_1(x)}} + \theta |v|^{\eta_1(x)-1}v|v|^{\gamma_1(x)}_{L^{s_1(x)}}  \ \mbox{in} \   \Omega, \\
-\mathcal{A}(x,|u|_{L^{r_2(x)}})\Delta_{p_2(x)}v &=&\lambda |v|^{\beta_2(x)-1}v|u|^{\alpha_2(x)}_{L^{q_2(x)}} + \theta |u|^{\eta_2(x)-1}u |u|^{\gamma_2(x)}_{L^{s_2(x)}}\ \mbox{in} \   \Omega, \\
u=v&=& 0\ \mbox{on}  \ \partial \Omega.
\end{array}
\right.\eqno{(S)_{\lambda,\theta}}
$$

The scalar and local version of $(S)_{\lambda,\theta}$ with $p(x)\equiv2$  and constant exponents was considered in the famous paper by Ambrosetti-Brezis-Cerami \cite{abc} in which a sub-supersolution
argument is used. In \cite{gel-giovany}, the problem $(S)_{\lambda,\theta}$ was studied with $p(x)\equiv2.$ The following result generalizes \cite[Theorem 7]{gel-giovany}.

\begin{theorem}Suppose that $r_i,p_i,q_i,s_i,\alpha_i$ and $\eta_i$ satisfy $(H)$ for $i=1,2 $ and  $\beta_i \in L^{\infty}(\Omega), i=1,2$  are nonnegative functions with $0<\alpha^{-}_i+\beta^{-}_i\leq\alpha^{+}_i+\beta^{+}_i<p^{-}_i-1,i=1,2$. Let $a_0,b_0>0$ positive numbers. The following assertions hold
	
\vspace{0.2cm}	

	\noindent{\bf $(A_{1})$} If $p^{+}_2-1<\eta^{-}_1+\gamma^{-}_1, \  p^{+}_1-1<\eta^{-}_2+\gamma^{-}_2$ and $\mathcal{A}(x,t)\geq a_{0}\;\text{in}\;\overline{\Omega}\times [0, b_{0}],$ then for each $\theta>0$ there exists $\lambda_{0}>0$ such that for each $\lambda\in(0,\lambda_{0})$ the  problem $(S)_{\lambda,\theta}$ has a positive solution $u_{\lambda,\theta}.$
	
\vspace{0.2cm}
	
	\noindent{\bf $(A_{2})$} Consider that $p^{+}_2-1<\eta^{-}_1+\gamma^{-}_1, \ p^{+}_1-1<\eta^{-}_2+\gamma^{-}_2$  and that the inequalities
	\begin{equation*}
	\frac{\beta^{+}_{1}}{p^{-}_{1}-1} + \frac{\alpha^{+}_{1} }{p^{-}_{2} -1} <1, \frac{\beta^{+}_{2}}{p^{-}_{2}-1} + \frac{\alpha^{+}_{2}}{p^{-}_{1}-1} <1
	\end{equation*}
hold.
	
Suppose that $0<\mathcal{A}(x,t)\leq a_{0}\;\mbox{in}\;\; \overline{\Omega}\times(0,\infty)$ and $\lim_{t\rightarrow\infty}\mathcal{A}(x,t)=b_{0}\;\text{uniformly in }\;\overline{\Omega}$.
	Then given $\lambda>0$, there exists $\theta_{0}>0$ such that for each $\theta \in(0,\theta_{0})$ the problem  $(S)_{\lambda,\theta}$ has a positive solution $u_{\lambda,\theta}.$
\end{theorem}

\begin{proof} Suppose that $(A_{1})$ occurs. Consider $z_{\lambda}\in W_{0}^{1,p_1(x)}({\Omega})\bigcap L^{\infty}(\Omega)$ and  $y_{\lambda}\in W_{0}^{1,p_2(x)}({\Omega})\bigcap L^{\infty}(\Omega)$   the unique solutions of \eqref{probl-linear-lambda}, where $\lambda\in(0,1)$ will be chosen before.

Lemma \ref{Fan} imply that for $\lambda>0$ small enough there exists a constant $K>1$ that does not depend on $\lambda$ such that  
\begin{equation}\label{desig1-p-supsol-concavo}
0<z_{\lambda}(x)\leq K\lambda^{\frac{1}{p^{+}_{1}-1}}\;\text{in}\;\Omega, 
\end{equation}
\begin{equation}\label{desig2-p-supsol-concavo}
0<y_{\lambda}(x)\leq K\lambda^{\frac{1}{p^{+}_{2}-1}}\;\text{in}\;\Omega.
\end{equation}
In order to construct $\overline{u}$ and $\overline{v}$ we will prove, for each $\theta >0,$ that there exists $\lambda_0>0$ such that 

\begin{equation}\label{c_1}
\frac{1}{a_0}\left( \lambda |z_{\lambda}|^{\beta_1(x)-1}z_{\lambda}|y_{\lambda}|^{\alpha_1(x)}_{L^{q_1(x)}} + \theta |w|^{\eta_1(x)-1}w|y_{\lambda}|^{\gamma_1(x)}_{L^{s_1(x)}}\right) \leq \lambda, \forall w \in [0,y_{\lambda}]
\end{equation}
and
\begin{equation}\label{c_2}
\frac{1}{a_0}\left( \lambda |y_{\lambda}|^{\beta_2(x)-1}y_{\lambda}|z_{\lambda}|^{\alpha_2(x)}_{L^{q_2(x)}} + \theta |w|^{\eta_2(x)-1}w |z_{\lambda}|^{\gamma_2(x)}_{L^{s_2(x)}}\right) \leq \lambda, \forall w \in [0,z_{\lambda}].
\end{equation}

Let 
\begin{equation}\label{overK-def}
\overline{K}:=\displaystyle\max_{i=1,2}\big\{ K^{\beta_{i}^{+}} |K|^{\alpha^{+}_{i}}_{L^{q_i(x)}},K^{\beta_{i}^{+}} |K|^{\alpha^{-}_{i}}_{L^{q_i(x)}}, K^{\eta_{i}^{+}} |K|^{\gamma^{+}_{i}}_{L^{s_i(x)}},  K^{\eta_{i}^{+}} |K|^{\gamma^{-}_{i}}_{L^{s_i(x)}}\big\}.
\end{equation}
Since $0<\alpha^{-}_{1} + \beta^{-}_{1} $ and $p^{+}_{2} -1 <\eta^{-}_{1} + \gamma^{-}_{1},$ there exists $\lambda_0>0$ such that

\begin{equation}\label{c_3}
\frac{1}{a_0} \left( \lambda^{\frac{p^{+}_{1} -1 + \beta^{-}_{1}}{p^{+}_{1}-1} + \frac{\alpha^{-}_{1}}{p^{+}_{2} -1}} \overline{K} + \theta \lambda^{\frac{\eta^{-}_{1} + \gamma^{-}_{1}}{p^{+}_{2} -1}}\overline{K}\right) \leq \lambda,
\end{equation}
for all $\lambda \in (0,\lambda_0).$

If necessary, one can consider a smaller $\lambda_0>0$ such that  $|y_{\lambda}|_{L^{r_1(x)}} \leq |K|_{L^{r_1(x)}} \lambda^{\frac{1}{p^{+}_{2}-1}} \leq b_0$ for all $\lambda \in (0,\lambda_0).$ Therefore $\mathcal{A}(x, |w|_{L^{r_1(x)}}) \geq a_0, w \in [0,y_{\lambda}].$ Thus from \eqref{desig1-p-supsol-concavo}, \eqref{desig2-p-supsol-concavo}  and \eqref{c_3} we  have that \eqref{c_1} holds. Then we can conclude that
\begin{equation}\label{prop1}
 - \Delta_{p_{1}(x)} z_{\lambda} \geq  \frac{1}{\mathcal{A}(x,|w|_{L^{r_1(x)}})} \left( \lambda {z_{\lambda}}^{\beta_1(x)} |y_\lambda|^{\alpha_1(x)}_{L^{q_1(x)}} + \theta w^{\eta_1(x)} |y_\lambda|^{\gamma_1(x)}_{L^{s_1(x)}}\right),
 \end{equation}
for all $w \in [0,y_\lambda].$

Consider also that $\lambda_0$ satisfies 
\begin{equation}\label{c_4}
\frac{1}{a_0} \left( \lambda^{\frac{p^{+}_{2} -1 + \beta^{-}_{2}}{p^{+}_{2}-1} + \frac{\alpha^{-}_{2}}{p^{+}_{1} -1}} \overline{K} + \theta \lambda^{\frac{\eta^{-}_{2} + \gamma^{-}_{2}}{p^{+}_{1} -1}}\overline{K}\right) \leq \lambda
\end{equation}
for all $\lambda \in (0,\lambda_0)$ and that $|z_\lambda|_{L^{r_2(x)}} \leq |K|_{L^{r_2(x)}}\lambda^{\frac{1}{p^{+}_{1}-1}} \leq b_0$ for all $\lambda \in (0,\lambda_0).$ Therefore $\mathcal{A}(x, |w|_{L^{r_2(x)}}) \geq a_0, w \in [0,z_{\lambda}].$ Thus from \eqref{desig1-p-supsol-concavo}, \eqref{desig2-p-supsol-concavo}  and \eqref{c_4} we  have that \eqref{c_2} holds. Then we can conclude that
\begin{equation}\label{prop2}
 - \Delta_{p_{2}(x)} y_{\lambda} \geq  \frac{1}{\mathcal{A}(x,|w|_{L^{r_2(x)}})} \left( \lambda {z_{\lambda}}^{\beta_2(x)} |z_\lambda|^{\alpha_2(x)}_{L^{q_2(x)}} + \theta w^{\eta_2(x)} |z_\lambda|^{\gamma_2(x)}_{L^{s_2(x)}}\right)
 \end{equation}
for all $w \in [0,z_\lambda].$

In order to construct $\underline{u}$ and $\underline{v},$ consider $\phi_i, \delta, \sigma, \mu$ as in the proof of Theorem \ref{teo-sublinear}. Using the inequalities $\alpha^{+}_{i} + \beta^{+}_{i} < p^{-}_{i}-1,i=1,2$ and repeating the arguments of Theorem \ref{teo-sublinear}, we have that exists a number $\mu>0$ such that $\mu \phi_1 \leq z_{\lambda}, \mu \phi_2 \leq y_{\lambda},$
$$- \Delta_{p_1(x)}(\mu \phi_1) \leq \lambda, $$
$$ - \Delta_{p_1(x)} (\mu\phi_1) \leq \frac{1}{\mathcal{A}(x,|w|_{L^{r_1(x)}})} \left(\lambda (\mu\phi_1)^{\beta_1(x)} |\mu \phi_1|^{\alpha_1(x)}_{L^{q_1(x)}} + \theta w^{\eta_1(x)} |\mu \phi_2|^{\gamma_1(x)}_{L^{s_1(x)}} \right), $$
for all $w \in [\mu \phi_2, y_{\lambda} ]$ and
$$- \Delta_{p_2(x)}(\mu \phi_2) \leq \lambda, $$
$$ -\Delta_{p_2(x)}(\mu \phi_2) \leq \frac{1}{\mathcal{A}(x,|w|_{L^{r_2(x)}})} \left( \lambda (\mu \phi_2)^{\beta_2(x)} |\mu \phi_1|^{\alpha_2(x)}_{L^{q_2(x)}}  + \theta w^{\eta_2(x)} |\mu \phi_1|^{\gamma_2 (x)}_{L^{s_2(x)}}\right),$$
for all $w \in [\mu \phi_2, z_{\lambda}].$ Then by Theorem \ref{theorem-to-(S)} we have the result.

Now we will consider the condition $(A_2).$ Consider $\phi_i,\delta$ and $\sigma_i,i=1,2$ as in the first part of the result and let $\lambda >0$ fixed. Since $\alpha^{+}_{i} + \beta^{+}_{i}< p^{-}_{i}-1,i=1,2$ there exist $\mu>0$ depending only on $\lambda $ such that
$$- \Delta_{p_i(x)}(\mu \phi_i) \leq 1 \ \textrm{and} \ -\Delta_{p_i(x)}(\mu \phi) \leq \frac{1}{a_0} \lambda (\mu \phi_i)^{\beta_i(x)} |\mu \phi_j|^{\alpha_i(x)}, $$
for $w \in L^{\infty}(\Omega)$ with $w \geq \mu\phi_j, i,j=1,2$ and $i \neq j.$ 

Let $M>0$ that will be chosen before and consider $z_M \in W^{1,p_1(x)}_{0}(\Omega) \cap L^{\infty}(\Omega)$ and $y_M \in W^{1,p_2(x)}_{0}(\Omega) \cap L^{\infty}(\Omega)$ solutions of

\begin{equation*}
\begin{aligned}
\left\{\begin{array}{rcl}
-\Delta_{p_1(x)}z_{M} &=&M\;\;\mbox{in} \;\;\Omega,\\
\vspace{.2cm}
z_M&=&0\;\;\mbox{on}\;\;\partial\Omega.
\end{array}
\right.
\end{aligned}
\hspace{2cm}
\begin{aligned}
\left\{\begin{array}{rcl}
-\Delta_{p_2(x)}y_{M} &=&M\;\;\mbox{in} \;\;\Omega,\\
\vspace{.2cm}
y_M&=&0\;\;\mbox{on}\;\;\partial\Omega.
\end{array}
\right.
\end{aligned}
\end{equation*}

If $M$ is large enough by Lemma \ref{Fan} there exists a constant $K>1$ that does not depend on $M$ such that
\begin{equation}\label{desig3-p-supsol-concavo}
0<z_{M}(x)\leq KM^{\frac{1}{p^{-}_{1}-1}}\;\text{in}\;\Omega, 
\end{equation}
\begin{equation}\label{desig4-p-supsol-concavo}
0<y_{M}(x)\leq KM^{\frac{1}{p^{-}_{2}-1}}\;\text{in}\;\Omega.
\end{equation}

In order to construct ${\overline{u}}_i ,{\overline{v}}_i,i=1,2$ we will show that exist $\theta_0 >0$ depending  on $\lambda$ with the following property: if we consider $\theta \in (0,\theta_0)$ then there will be a constant $M$ depending only  on $\lambda$ and $\theta$   satisfying
\begin{equation}\label{c_5}
 M \geq \frac{1}{\mathcal{A}(x,|w|_{L^{r_1(x)}})} \left(  \lambda {z_M}^{\beta_1(x)} |y_M|^{\alpha_1(x)}_{L^{q_1(x)}} + \theta {w}^{\eta_1(x)} |y_M|^{\gamma_1(x)}_{L^{s_1(x)}} \right), w \in [\mu \phi_2, y_M]
\end{equation}
and
\begin{equation}\label{c_6}
 M \geq \frac{1}{\mathcal{A}(x,|w|_{L^{r_2(x)}})} \left(  \lambda {y_M}^{\beta_2(x)} |z_M|^{\alpha_2(x)}_{L^{q_2(x)}} + \theta {w}^{\eta_2(x)} |z_M|^{\gamma_2(x)}_{L^{s_2(x)}} \right), w \in [\mu \phi_1, z_M].
\end{equation}

Since $\mathcal{A}$ is continuous and $\displaystyle\lim_{t \rightarrow + \infty} \mathcal{A}(x,t)=b_0>0$ uniformly in $\Omega$ there is  $a_1 >0$ large enough such that $\mathcal{A}(x,t) \geq \frac{b_0}{2}$ in $\overline{\Omega} \times (a_1, +\infty).$ Define
$$m_\lambda:= \{ \mathcal{A}(x,t): (x,t) \in \overline{\Omega} \times [\min\{|\mu \phi_1|_{L^{r_1(x)}}, |\mu \phi_2|_{L^{r_2(x)}}\}, a_1]\} $$
and $\mathcal{A}_{\lambda}:= \min\{m_\lambda, \frac{b_0}{2} \}$. Then $\mathcal{A}(x,t) \geq \mathcal{A}_{\lambda}$ in $\overline{\Omega} \times [\min \{|\mu \phi_1|_{L^{r_1(x)}}, |\mu \phi_2|_{L^{r_2(x)}}\},\infty).$ Thus $\mathcal{A}_{\lambda} \leq \mathcal{A}(x,|w|_{L^{r_1(x)}}) \leq a_0$ for all $w \in L^{\infty}(\Omega)$ with $ \mu \phi_1 \leq w$ or $\mu \phi_2 \leq w$.
Note that from \eqref{desig3-p-supsol-concavo} and \eqref{desig4-p-supsol-concavo} the inequalities \eqref{c_5} and \eqref{c_6} hold if we have simultaneously the inequalities

$$ \frac{1}{\mathcal{A}_{\lambda}} \left( \lambda \overline{K} M^{\frac{\beta^{+}_{1}}{p^{-}_{1}-1} + \frac{\alpha^{+}_{1}}{p^{-}_{2}-1}} + \theta \overline{K} M^{\frac{\eta^{+}_{1} + \gamma^{+}_{1}}{p^{-}_{2}-1}}\right) \leq M$$
and
$$ \frac{1}{\mathcal{A}_{\lambda}} \left( \lambda \overline{K} M^{\frac{\beta^{+}_{2}}{p^{-}_{2}-1} + \frac{\alpha^{+}_{2}}{p^{-}_{1}-1}} + \theta \overline{K} M^{\frac{\eta^{+}_{2} + \gamma^{+}_{2}}{p^{-}_{1}-1}}\right) \leq M,$$
where $\overline{K}$ is given by \eqref{overK-def}. In order to obtain such inequalities we will study the inequality
\begin{equation}\label{equiv-f}
\frac{1}{\mathcal{A_\lambda}}\left( \lambda \overline{K} M^{\rho-1} + \theta \overline{K}M^{\tau-1}\right) \leq 1
\end{equation}
where 
$$\rho:= \max \left\{ \frac{ \beta^{+}_{1}}{p^{-}_{1}-1} + \frac{\alpha^{+}_{1}}{p^{-}_{2} -1 }, \frac{\beta^{+}_{2}}{p^{-}_{2} -1}  + \frac{\alpha^{+}_{2}}{p^{-}_{1}-1}\right\}$$
and
$$\tau:= \max \left\{\frac{\eta^{+}_{1} + \gamma^{+}_{1}}{p^{-}_{2}-1}, \frac{\eta^{+}_{2} + \gamma^{+}_{2}}{p^{-}_{1} -1} \right\}. $$
Define 
$$\Psi_{\lambda,\theta}(M):= \frac{\lambda \overline{K}}{\mathcal{A}_{\lambda}} M^{\rho-1} + \frac{\theta \overline{K}}{\mathcal{A}_{\lambda}} M^{\tau -1} , M >0.$$
Since $0< \rho<1$ and $\tau>1$ we have $\displaystyle\lim_{M \rightarrow 0^+}\Psi_{\lambda,\theta}(M) = \displaystyle\lim_{M \rightarrow + \infty }\Psi_{\lambda,\theta}(M) = +\infty. $
Note that ${\Psi_{\lambda,\theta}}^{\prime}(M) = 0$ if, and only if
\begin{equation}\label{minimum}
 M=M_{\lambda,\theta}:= \left( \frac{\lambda}{\theta}\right)^{\frac{1}{\tau - \rho}}c
\end{equation}
where $c:= \left( \frac{1-\rho}{\tau -1}\right)^{\frac{1}{\tau - \rho}}.$ From the above properties of $\Psi_{\lambda,\mu}$ we have that the global minimal of $\Psi_{\lambda, \theta}$ occurs at $M_{\lambda,\theta}$. The inequality \eqref{equiv-f} is equivalent to find
$M_{\lambda,\theta} > 0$ such that $\Psi_{\lambda,\theta}(M_{\lambda,\theta})\leq 1$. By \eqref{minimum}, we have $\Psi_{\lambda,\theta}(M_{\lambda,\theta}) \leq 1,$ if and only if 

\begin{equation}\label{minimum-2}
\frac{\lambda \overline{K}}{\mathcal{A}_{\lambda}}\left( \frac{\lambda}{\theta}\right)^{\frac{\rho-1}{\tau - \rho}}c^{\rho-1} + \theta^{1 - \left(\frac{\tau-1}{\tau-\rho}\right)} \frac{\overline{K}}{\mathcal{A}_{\lambda}}\lambda^{\frac{\tau-1}{\tau - \rho}}c^{\tau-1} \leq 1. 
\end{equation}
Thus from \eqref{minimum} and \eqref{minimum-2}, we have that given $\lambda>0$ there exists $\theta_0>0$ such that for each $\theta \in (0,\theta_0)$ there exists $M_{\lambda,\theta}$ such that
$$ M_{\lambda ,\theta} \geq 1 \ \textrm{and} \ \frac{1}{\mathcal{A}_{\lambda}} \left( \lambda \overline{K} {M_{\lambda,\theta}}^{\rho-1} + \theta \overline{K}{M_{\lambda,\theta}}^{\tau-1}\right) \leq 1.$$
Therefore 
$$ -\Delta_{p_1(x)} z_M  \geq \frac{1}{\mathcal{A}_{\lambda}} \left( \lambda {z_M}^{\beta_1(x)} |y_M|^{\alpha_1(x)}_{L^{q_1(x)}} + \theta {w}^{\eta_1(x)} |y_M|^{\gamma_1(x)}_{L^{s_1(x)}}\right) \ \textrm{in} \ \Omega,$$
for all $w \in [\mu \phi_2 , y_M]$ and
$$ -\Delta_{p_2(x)} y_M  \geq \frac{1}{\mathcal{A}_{\lambda}} \left( \lambda {y_M}^{\beta_2(x)} |z_M|^{\alpha_2(x)}_{L^{q_2(x)}} + \mu {w}^{\eta_2(x)} |z_M|^{\gamma_w(x)}_{L^{s_w(x)}}\right) \ \textrm{in} \ \Omega,$$
for all $w \in [\mu \phi_1, z_M].$
Note that for $\theta_0 > 0$ small enough, we have for $\theta \in (0,\theta_0)$ that
$$- \Delta_{p_1(x)} (\mu \phi_1)  \leq 1 \leq M_{\lambda, \theta_0} \leq M_{\lambda,\theta},$$
because $M_{\lambda,\theta} \rightarrow +\infty$ as $\theta \rightarrow 0^+$ and $\theta \longmapsto M_{\lambda, \theta}$ is decreasing. Similarly, we have $-\Delta_{p_2(x)}(\mu \phi_2) \leq M_{\lambda, \theta_0} \leq M_{\lambda, \theta} $ for all $\theta \in (0,\theta_0)$, if $\theta_0$ is small enough. The weak maximum principle imply that $\mu \phi_1 \leq z_M$ and $\mu \phi_2 \leq y_M.$ The result is proved.

\end{proof}

\subsection{A generalization of the logistic equation:}

In the previous sections, we considered at least one of the conditions $\mathcal{A}(x,t)\geq a_{0}>0$ or $0<\mathcal{A}(x,t)\leq a_{\infty}, t>0.$ In this last section we study a generalization of the classic logistic equation where the function $\mathcal{A}(x,t)$  can satisfy
$$\mathcal{A}(x,0)\geq0,\;\;\;\lim_{t\rightarrow0^{+}}\mathcal{A}(x,t)=\infty
\;\;\;\mbox{and}\;\;\;\lim_{t\rightarrow + \infty}\mathcal{A}(x,t)=\pm\infty.$$

We will consider the problem
$$\left\{\begin{array}{rcl}
-\mathcal{A}(x,|v|_{L^{r_1(x)}})\Delta_{p_1(x)} u&=&\lambda f_1(u)|v|_{L^{q_1(x)}}^{\alpha_1(x)}\;\;\mbox{in}\;\;\Omega,\\
-\mathcal{A}(x,|u|_{L^{r_2(x)}})\Delta_{p_2(x)} v&=&\lambda f_2(v)|u|_{L^{q_2(x)}}^{\alpha_2(x)}\;\;\mbox{in}\;\;\Omega,\\
\vspace{.2cm}
u=v&=&0\;\;\mbox{on}\;\;\partial\Omega.
\end{array}
\right. \eqno{(P)_{\lambda}}$$

We suppose that there are numbers $\theta_i>0,i=1,2$ such that the functions $f_i:[0,\infty)\rightarrow\mathbb{R}$ satisfy the conditions:

\noindent{\bf $(f_{1})$} $\;f_i\in C^{0}([0,\theta_i],\mathbb{R}),i=1,2;$

\noindent{\bf $(f_{2})$} $\;f_i(0)=f_i(\theta_i)=0,\;\;f_i(t)>0\; \text{in}\;(0,\theta_i)$ for $i=1,2.$

Problem  $(P)_{\lambda}$ is a generalization of the problems studied in  \cite{chipot-correa, chipot-roy, gel-giovany}. The next result generalizes \cite[Theorem 8]{gel-giovany}.

\begin{theorem}
	Suppose that  $r_i,p_i,q_i,\alpha_i$ satisfy $(H).$ Consider also that $f_i,i=1,2$ satisfies $(f_{1}), (f_{2})$ and that	$\mathcal{A}(x,t)>0$ in $\overline{\Omega}\times\bigl(0,\max\{|\theta_1|_{L^{r_2(x)}},|\theta_2|_{L^{r_1(x)}}\}\bigl].$ Then there exists $\lambda_{0}>0$ such that $\lambda\geq\lambda_{0},$ $(P)_{\lambda}$ has a positive solution.
\end{theorem}
\begin{proof}
Consider the functions $\widetilde{f}_i(t)=f_i(t)$ for $t\in[0,\theta_i],$ and $\widetilde{f}_i(t)=0$ for $t\in\mathbb{R}\setminus[0,\theta_i].$ The functional 
\begin{align*}
J_{\lambda}(u,v)&=\int_{\Omega}\frac{1}{p_1(x)}|\nabla u|^{p_1(x)}dx-\lambda\int_{\Omega}\widetilde{F}_1(u)dx + \int_{\Omega}\frac{1}{p_2(x)}|\nabla v|^{p_2(x)}dx-\lambda\int_{\Omega}\widetilde{F}_2(v)dx\\
&:= J_{1,\lambda}(u) +  J_{2,\lambda}(v),
\end{align*}
where $\widetilde{F}_i(t)=\int_{0}^{t}\widetilde{f}_i(s)ds$ is of class $C^{1}(W_{0}^{1,p_1(x)} \times W_{0}^{1,p_2(x)}(\Omega),\mathbb{R})$ and $W_{0}^{1,p_1(x)} \times W_{0}^{1,p_2(x)}(\Omega) $ is a Banach space endowed with the norm
$$|(u,v)|:= \max \left\{| \nabla u|_{p_1(x)}, | \nabla v|_{p_2(x)}\right\}.$$


Since $|\widetilde{f}(t)|\leq C$ for $t\in\mathbb{R}$ we have that $J$ is coercive. Thus $J$ has a minimum $(z_{\lambda},w_{\lambda}) \in W^{1,p_1(x)}_{0}(\Omega) \times W^{1,p_2(x)}_{0}(\Omega)$ with
\begin{equation}\label{P1}
\left\{\begin{array}{rcl}
-\Delta_{p_1(x)} z_{\lambda}&=&\lambda \widetilde{f}_1(z_\lambda)\;\;\mbox{in}\;\;\Omega,\\
\vspace{.2cm}
z_\lambda&=&0\;\;\mbox{on}\;\;\partial\Omega,
\end{array}
\right.
\end{equation}
and
\begin{equation}\label{P2}
\left\{\begin{array}{rcl}
-\Delta_{p_2(x)} w_{\lambda}&=&\lambda \widetilde{f}_2(w_\lambda)\;\;\mbox{in}\;\;\Omega,\\
\vspace{.2cm}
w_\lambda&=&0\;\;\mbox{on}\;\;\partial\Omega.
\end{array}
\right.
\end{equation}
Note that the unique solutions of $\eqref{P1}$ and $\eqref{P2}$ are given by the minimum of functionals $J_{1,\lambda}$ and $J_{2,\lambda}$ respectively.

Consider a function $\varphi_0 \in W^{1,p_i(x)}_{0}(\Omega), i=1,2$ with $\tilde{F}_{i}(\varphi_0)>0,i=1,2.$
Define $(z_0,w_0):= (z_{\tilde{\lambda}_{0}}, w_{\tilde{\lambda}_{0}}),$ where $\tilde{\lambda}_0$ satisfies
$$\int_{\Omega}\frac{1}{p_i(x)}|\nabla \varphi_{0}|^{p_i(x)} dx <\widetilde{\lambda}_{0}\int_{\Omega}\widetilde{F}_i(\varphi_{0}) dx,i=1,2.$$
We have  $J_{1,\tilde{\lambda}_0}(z_0) \leq J_{1,\tilde{\lambda}_0}(\varphi_0)<0$ and also that $J_{2,\tilde{\lambda}_0}(z_0)<0$. Therefore $z_0 \neq 0$ and $w_0 \neq 0$. Since $-\Delta_{p_1(x)}z_0$ and $-\Delta_{p_2(x)} w_0 $ are nonnegative, we have $z_0,w_0 >0$ in $\Omega.$ Note that by \cite[Theorem 4.1]{fan-zao}  and  \cite[Theorem 1.2]{fan-regular}, we obtain that $z_{0},w_0\in C^{1,\alpha}(\overline{\Omega})$ for some $\alpha \in (0,1].$

Using the test function $\varphi=(z_{0}-\theta_1)^{+}\in W_{0}^{1,p_1(x)}(\Omega)$ in $\eqref{P1}$ we get
$$\int_{\Omega}|\nabla z_{0}|^{p_1(x)-2}\nabla z_{0}\nabla(z_{0}-\theta_1)^{+} dx=\widetilde{\lambda}_{0}\int_{\{z_{0}>\theta\}}\widetilde{f}_1(z_{0})(z_{0}-\theta_1) dx=0.$$
Therefore
\begin{eqnarray*}
	\int_{\{z_{0}>\theta\}}\bigl<|\nabla z_{0}|^{p(x)-2}\nabla z_{0}-|\nabla\theta_1|^{p_1(x)-2}\nabla \theta_1,\nabla(z_{0}-\theta_1)\bigl> dx=0,
\end{eqnarray*}
which imply  $(z_{0}-\theta_1)_{+}=0$ in $\Omega.$ Thus $0<z_{0}\leq\theta_1.$ A similar reasoning provides $0<w_0 \leq \theta_2.$

Note that there is a constant $C>0$ such that $|z_{0}|_{L^{q_1(x)}}^{\alpha_1(x)},|w_0|^{\alpha_2(x)}_{L^{q_2(x)}}\geq C.$ Define $\mathcal{A}_{0}:=\max\big\{\mathcal{A}(x,t):(x,t)\in\overline{\Omega}\times[\min \{ |z_0|_{L^{r_2(x)}}, |w_0|_{L^{r_1(x)}}\}, \max\{|\theta_1|_{L^{r_2(x)}},|\theta_2|_{L^{r_1(x)}} \} \big\}$ and $\mu_{0}=\frac{\mathcal{A}_{0}}{C}.$ Then, we have
$$-\Delta_{p_1(x)}z_{0}=\widetilde{\lambda}_{0} f_1(z_{0})=\frac{1}{\mathcal{A}_{0}}\widetilde{\lambda}_{0}\mu_{0}f_1(z_{0})|w_{0}|_{L^{q_1(x)}}^{\alpha_1(x)}\frac{\mathcal{A}_{0}}{\mu_{0}|z_{0}|_{L^{q_1(x)}}^{\alpha_1(x)}}\leq \frac{1}{\mathcal{A}_{0}}\widetilde{\lambda}_{0}\mu_{0}f_1(z_{0})|w_{0}|_{L^{q_1(x)}}^{\alpha_1(x)}.$$

Thus for each $\lambda\geq \lambda_0:=\widetilde{\lambda}_{0}\mu_{0}$ and $w\in[w_0,\theta_2],$ we get
$$-\Delta_{p_1(x)}z_{0}\leq\frac{1}{\mathcal{A}(x,|w|_{L^{r_1(x)}})}\lambda f_1(z_{0})|w_{0}|_{L^{q_1(x)}}^{\alpha_1(x)}.$$

If necessary, we can consider a bigger $\lambda_0>0$ such that
$$- \Delta_{p_2(x)} w_0 \leq \frac{1}{\mathcal{A}(x,|w|_{L^{r_2(x)}})} \lambda f_2(w_0)|z_0|^{\alpha_2(x)}_{L^{q_2(x)}}, $$
for all $\lambda \geq \lambda_0$ and $w \in [z_0,\theta_1].$

Since $f_i(\theta_i)=0,i=1,2,$ we have that $(z_0,\theta_1)$ and $(w_0,\theta_2)$ are sub-supersolutions pairs for $(P)_{\lambda}$.  The result is proved.
\end{proof}

\begin{remark}
We would like to point that out that is possible to use the functions $\phi_i,i=1,2$ from the proof of Theorem \ref{teo-sublinear} to consider  problem $(P)_{\lambda}.$  However, more restrictions on the functions $p_i,f_i,i=1,2$ are needed.
\end{remark}


\end{document}